\documentclass[graybox]{svmult}

\usepackage{stmaryrd}
\usepackage{amssymb} % amsmath, amsthm}
\usepackage{bussproofs}
\usepackage{rotfloat}
\usepackage{linguex}
\usepackage[x11names]{xcolor}
\usepackage{colortbl}

\usepackage{type1cm}        % activate if the above 3 fonts are
\usepackage{makeidx}         % allows index generation
\usepackage{graphicx}        % standard LaTeX graphics tool
\usepackage{multicol}        % used for the two-column index
\usepackage[bottom]{footmisc}% places footnotes at page bottom
\usepackage{newtxtext}       % 
\usepackage{newtxmath}       % selects Times Roman as basic font

\usepackage{hyperref}
\hypersetup{colorlinks=true,allcolors=blue}

\EnableBpAbbreviations

\floatstyle{boxed}
\restylefloat{figure}

\DeclareMathOperator{\glb}{glb}
\DeclareMathOperator{\lub}{lub}

\newcommand{\comment}[1]{}
\newcommand{\half}{.5}
\newcommand{\tuple}[1]{\ensuremath{\langle #1 \rangle}}
\newcommand{\G}{\Gamma}
\newcommand{\Del}{\Delta}
\newcommand{\sqq}[2]{\ensuremath{#1 \Yright #2}}
\newcommand{\SQQ}[2]{\ensuremath{#1 \tstile #2}}
\newcommand{\vs}{\vskip 5mm}
\newcommand{\hs}{\hskip 5mm}
\newcommand{\LLl}[1]{\LeftLabel{\scriptsize #1:\quad}}
\newcommand{\AXx}[2]{\Axiom$ #1 \fCenter #2$}
\newcommand{\UIx}[2]{\UnaryInf$ #1 \fCenter #2$}
\newcommand{\BIx}[2]{\BinaryInf$ #1 \fCenter #2$}

\newcommand{\val}[1]{\mbox{$[\![  #1 ]\!]$}}

\newcommand{\tstile}{\vDash}
\newcommand{\Sc}{\tstile_{\text{S}}}
\newcommand{\STc}{\tstile_{\text{ST}}}
\newcommand{\Cc}{\tstile_{\text{C}}}

\newcommand{\STsim}{\text{ST$_{\sim}$}}

\newcommand{\para}{\mathcal{P}}
\newcommand{\pc}[1]{\tstile_{#1}}
\newcommand{\pcp}{\pc{\para}}
\newcommand{\psc}[1]{\tstile^{\sim}_{#1}}
\newcommand{\STsc}{\psc{\text{ST}}}
\newcommand{\pscp}{\psc{\para}}

\newcommand{\iscrisped}{\preccurlyeq}

\theoremstyle{fact}
\newtheorem{fact}{Fact}
\newtheorem{lem}{Lemma}
\newtheorem{thm}{Theorem}

\theoremstyle{definition}
\newtheorem{defn}{Definition}

\def\fCenter{\,\vdash\,}

\newcommand{\addPE}[1]{{\leavevmode\color{black}#1}}
\newcommand{\addEP}[1]{{\leavevmode\color{black}#1}}

\makeindex             % used for the subject index
\begin{document}

\title*{Tolerance and degrees of truth\thanks{This is the penultimate version of a paper to appear in a volume edited by M. Petrolo and G. Venturi. Please cite the published version once available. Acknowledgments appear in a separate section after the main text.}}
% Use \titlerunning{Short Title} for an abbreviated version of
% your contribution title if the original one is too long
\author{Pablo Cobreros, Paul \'{E}gr\'{e}, David Ripley, Robert van Rooij}
% Use \authorrunning{Short Title} for an abbreviated version of
% your contribution title if the original one is too long
\institute{Pablo Cobreros \at University of Navarra,
\email{pcobreros@unav.es}
\and Paul \'{E}gr\'e \at Institut Jean-Nicod (CNRS, EHESS, ENS-PSL University), \email{paul.egre@ens.psl.eu}
\and David Ripley \at Monash University,
\email{davewripley@gmail.com}
\and Robert van Rooij \at ILLC, Amsterdam,
\email{R.A.M.vanRooij@uva.nl}
}
%
% Use the package "url.sty" to avoid
% problems with special characters
% used in your e-mail or web address
%
\maketitle

\abstract{This paper explores the relations between two logical approaches to vagueness: on the one hand the fuzzy approach defended by \cite{smith:vdt}, and on the other the strict-tolerant approach defended by \cite{cervr:tcs}. Although the former approach uses continuum many values and the latter \addPE{implicitly} four, we show that both approaches can be subsumed under a common three-valued framework.
In particular, we {defend the claim} that Smith's continuum many values are not needed to solve what Smith calls `the jolt problem', and we show that they are not needed for his account of logical consequence either.
Not only are three values enough to satisfy Smith's central desiderata, but they also allow us to internalize Smith's closeness principle in the form of a tolerance principle at the object-language. The reduction, we argue, matters for the justification of many-valuedness in an adequate theory of vague language.}%\bigskip
%

%\noindent \textbf{Keywords.} Vagueness, tolerance, fuzzy logic, many-valued logic, logical consequence, sorites paradox, strict-tolerant logic
%\end{abstract}

\newpage

\section{Introduction}

In this paper, we explore the relations between two logical approaches to vagueness:\
the degree-theoretic approach of \cite{smith:vdt} and the strict-tolerant account of vagueness originally laid out in \cite{cervr:tcs}. 
At first glance, these approaches look quite different: the former is based on a fuzzy logic, that is on continuum many truth values, whereas the latter is implicitly based on a four-valued logic.
Conceptually, the approaches differ still further: on Smith's approach, the principle of \emph{tolerance} is rejected (the principle whereby anyone imperceptibly shorter than a tall person must count as tall), to be replaced with a weaker principle of \emph{closeness} (according to which for any two persons $x$ and $y$ with imperceptibly distinct heights, the semantic values of ``$x$ is tall'' and ``$y$ is tall'' must be close). Smith moreover articulates both principles as metalinguistic constraints. 
On the strict-tolerant logic of vagueness developed in \cite{cervr:tcs}, on the other hand, the principle of tolerance is {stated} in the object-language, and it is valid. 

Our point of departure in this paper is that there is less to these differences than meets the eye. The two approaches have much in common. In particular, they can both be cast into a common three-valued framework \cite{cervr:vtpc}, as we proceed to explain below.
This is because the accounts both work with a common structure: they have it that whenever all the premises of an argument hold to some particular strong standard, the conclusion must hold to some particular weaker standard (a form of what we call ``permissive consequence'' in \cite{cervr:vtpc}).
It is this common structure that is directly captured by a three-valued framework.
Moreover, despite the kinds of models these approaches use, both stick to full first-order classical logic, validating every classically-valid argument.
\cite{cervr:tcs} but not \cite{smith:vdt} extends the object language to allow for the principle of tolerance to be stated.{\footnote{{Strict-tolerant logic conservatively extends classical logic (it is fully classical over its $\sim$-free fragment), but it is nontransitive over its $\sim$-full fragment. The loss of transitivity is arguably a non-classical feature of ST-logic. See \cite{cervr:tcs} and below for discussion.}}}

In what follows, we consider extended versions of both approaches, in order to bridge the gap between them. One of our goals in this paper is logical: we introduce a family of consequence relations, which we call \emph{parameterized consequence} relations, which basically subsume both Smith's consequence and strict-to-tolerant consequence ({abbreviated ST-consequence}) as particular cases. We use the framework to show that strict-to-tolerant consequence in a sense occupies a unique position among this family of consequence relations when the language is sufficiently expressive to accommodate tolerance principles for vague predicates.  Our main goal, however, is more philosophical, and concerns the role and the number of truth values in an adequate theory of vague language. 

Two roles are sometimes distinguished for truth values: a \emph{referential} role, used to assign values to sentences, and a logical or \emph{inferential} role, concerned with entailment relations between sentences (\cite{suszko1977fregean, malinowski:cilmv, shramko2011truth}). Smith's theory and the strict-tolerant theory do not use the same number of truth values to serve as references for sentences, but our argument will show that the richness of Smith's referential apparatus can be cut down to three values when it comes to capturing central inferential principles about vagueness. There is a clear and precise sense, therefore, in which the strict-tolerant approach is canonical for the kind of treatment of vagueness advocated by Smith. This reduction is not meant to offer a full-fledged theory of the relation between referential and inferential many-valuedness for vague language, but we take it to be a step in that direction, not least because vagueness is an area in which 
the introduction of intermediate truth values between 0 and 1 {has been a source of controversy and remains in need of further justification (see \cite{williamson:v, haack1996deviant} and \cite{smith:vdt} for rival views)}.

%Smith's theory uses infinitely many values as references of sentences, but our argument shows that only three values are needed at the inferential level to capture his consequence relation. By contrast, the strict-tolerant basically uses three values at both levels, referential and inferential: on our approach, the third value is assigned to a vague predicate to capture borderline status, and the same third value is used to validate the tolerance principle. Smith argues that any finite assignment of values to vague sentences will involve `jolts' between distinct semantic values, but our contention is that a jolt-free approach is achieved first and foremost at the inferential level in the form of tolerance principles. There is a clear and precise sense, therefore, in which the strict-tolerant approach is canonical for the kind of treatment of vagueness advocated by Smith.

%The core of our argument is that three values can do a lot more for an adequate theory of vagueness than is standardly assumed. 
%We will discuss the logical situation involving these approaches with an eye to some philosophical upshots.
%The main one, we think, is this: despite Smith's claims to the contrary, if his approach to vagueness is a workable one, then so too are various three-valued approaches along the lines we will present.

Our paper is structured as follows. 
The next section gives a brief review {and a philosophical discussion} of Smith's account of closeness and tolerance, rebutting his arguments against three-valued approaches to vagueness. 
\S\ref{smithst} gives a more detailed formal comparison between Smith's account and the strict-tolerant account. 
\S\ref{paracons} introduces the notion of parameterised consequence and shows how to embed both Smith's consequence and ST-consequence under that scheme. 
In \S\ref{discussion}, finally, we draw more general lessons from this comparison, in particular regarding the truth values needed to study vagueness.

%%%%%

\section{Tolerance and closeness} \label{jolt}%: are three valued not enough?}

Vague predicates seem to exhibit a phenomenon known as {\em tolerance} \cite{wright:ocvp, kamp:poh}. 
For example, consider the vague predicate ``young'': there seem to be differences in age too small to matter for youngness (such as a difference of a nanosecond for age in humans, at least in most contexts).
Tolerance is the claim that this seeming is correct: that there really are differences too small to matter.
In general, for a predicate $P$, $P$-similarity (which we write $\sim_P$) is a relation that holds between things when, according to tolerance, they are too similar in whatever respects matter for $P$ for it to be the case that $P$ applies to one of them but not the other. 
The principle of tolerance may then be stated as follows:

\ex. \label{tol} $\forall x \forall y (Px \wedge x\sim_{P} y \rightarrow Py)$\quad (Tolerance)

This is a notoriously problematic principle in two-valued classical logic, {for conjoined with the existence of a soritical series}, it gives rise to the sorites paradox. 
For vague $P$, it is easy to imagine a list of individuals, each $P$-similar to the next, but where the first is obviously $P$ and the last obviously not $P$.
The way this problem is evaded in classical logic usually involves a rejection of the tolerance principle.\footnote{Bare rejection of the tolerance principle, of course, doesn't give much of a useful theory.
Because of this, rejection of tolerance is often paired with some explanation of its intuitive appeal.
These explanations might involve, for example, epistemology \cite{sorensen:vac, williamson:v}, context-sensitive concepts \cite{fara:ss, raffman:vcr}, or pragmatic restrictions \cite{manor:solving, gaifman:vtcl, rooij2011implicit, pagin:vdr, gt:slpdpv}.
Tolerance can be maintained, at least in some cases, by denying that there can be such a chain of $P$-similar individuals in the first place; see \cite{fara:pcas}.}
Rejecting tolerance in this two-valued setting, however, leaves us with the idea that in every sorites series,  there are at least two individuals $a$ and $b$ that are very close in $P$-relevant respects, but where $Pa$ and $Pb$ are still assigned opposite truth values.\footnote{The needed quantifier move here---from $\neg\forall$ to $\exists\neg$---is not valid in intuitionistic logic.
We don't discuss intuitionistic approaches to the sorites here; see \cite{rw:hpt, wright:quandary, rumfitt:bst} for more.}

This violation of tolerance leaves us with a sudden jolt in truth values as we proceed along a sorites series; Smith thus calls it `the jolt problem'.
(On Smith's view, by contrast, the core of vagueness lies in the absence of such jolts:  gradual modifications of the features relevant for a property should be matched by gradual modifications of the  truth values assigned to the claim that the property applies.)
Consequently, neither tolerance nor its negation can provide a satisfactory account of vagueness in this two-valued setting: tolerance because it leads to paradox, and the negation of tolerance because of the jolt problem. 

According to Smith, the proper way to evade this dilemma is to abandon two-valued logic and replace the principle of tolerance by a principle of \emph{closeness}. Given a vague predicate $P$, closeness states that if two objects $a$ and $b$ are $P$-similar, then the sentences $Pa$ and $Pb$ should have truth values that are very close to each other. 
Formally, this may be represented as follows, letting $\sim_{T}$ stand for closeness between truth values, and $\val{Pa}$ for the truth value of $Pa$:

\ex. If $a\sim_{P} b$, then $\val{Pa}\sim_{T} \val{Pb}$\qquad (Closeness)

Closeness is the leading principle behind the introduction of degrees of truth in Smith's approach. 
Two degrees of truth seem obviously inadequate to accommodate closeness, so more than two truth values are required. 
But how many more, if closeness is to be secured?
Smith's answer is as follows:

\begin{quote}

I do not know exactly how many degrees of truth we need in order to accommodate Closeness. The point is simply that we need a significant number of them. [\ldots] As far as accommodating vagueness goes, we might have a large finite number of degrees of truth [\ldots] or we might have continuum-many degrees of truth (as in the fuzzy picture) (p.\ 190).

\end{quote} 

Smith's choice is to have continuum many degrees of truth, though Smith admits that there are no conclusive reasons against having a finite number of degrees of truth.\footnote{In particular, Smith rejects arguments based on the seeming arbitrariness of any particular finite number; see p.\ 190.% point to mention ? Smith pg. 64 fn. 40 on not using the rationals between 0 and 1 for reasons of completeness. But could he use denumerably many values and still have completeness?
} 
An option Smith explicitly rejects, however, is working with only three degrees of truth \cite[p.\ 186]{smith:vdt}, \cite[p.\ 178]{smith:vac}. 
According to Smith, a third-value view will necessarily suffer from the jolt problem. \addPE{Smith's reason to reject trivalent approaches is as follows (p. 186):

\begin{quote}
If one sentence is True and another False, then they are as far apart as can
be in respect of truth --- and furthermore, they are in an absolute sense \emph{very
far apart} in respect of truth. Given that Truth and Falsity are poles apart
in this way, no third truth status can be \emph{very close} to \emph{both} of them.
\end{quote}

Smith's argument is meant to be fully general, that is, it ought not to depend on the interpretation of the third truth value. What about the case, however, in which the third truth value is interpreted as ``Both-True-and-False''? (For approaches to vagueness that take this route, see \cite{cobreros:vs, priest:vi, ripley:sos, wrphc:tg}.) Isn't this third truth value close to ``True'' and close to ``False''? 

Take two $P$-similar objects $a$ and $b$ such that $Pa$ gets value 1, and $Pb$ the value \half: since this means that $Pa$ is True, and that $Pb$ is True and False, there no longer is any jolt in this case, since the value assigned to $Pb$ retains some element of the value assigned to $Pa$. 
These claims match in an important respect; they are both true.
The same holds if $Pc$ is assigned \half\ and $Pd$ is assigned 0. 
These claims match in an important respect; they are both false.
The only jolt would be a situation in which $a$ and $b$ are $P$-similar, but such that $Pa$ gets 1 and $Pb$ 0---and this is exactly what does not arise, on such an approach.
So, arguably, a glut-based theory of vagueness does not predict any jolts.\footnote{{See \cite{egre:vdtreview} for a preliminary version of this objection. Here, however, we focus on Smith's topological argument against three values. We thank an anonymous for urging us to do so.}}

%Smith's  objection would be that 

\addEP{This argument would not suffice to convince Smith, however. For Smith's point is that regardless of its dialetheist definition,} a value such as Both-True-and-False cannot be \emph{very close} both to True and to False. Thus, Smith writes (p. 186):

\begin{quote}

For if
one thing is \textit{very} similar to each of two other things in some respect, then
those two things must at the very least be \textit{reasonably} similar to one another
in that respect --- yet Truth and Falsity are not similar at all in respect of
truth. Thus, to the extent that a sentence is very close to True, it is not very
close to False, and vice versa.

\end{quote}

Our response here is that True and False can actually be taken to be ``reasonably similar'', though we need to be careful about the relevant respect. While Smith writes that they are ``not similar at all in respect of \emph{truth}'', we can observe that they are similar in respect of being \emph{truth-values}, consistently with being as far apart as can be along that dimension. Consider the following analogy: ``black'' and ``white'' are certainly not similar at all ``in respect of (their proximity to) black'', but they are similar in respect of being achromatic colors. %\footnote{\addEP{Which we take to form a natural category for vision. Our claim is \emph{not} that any two objects, like a shoe and a fish, are going to be similar in respect of belonging to the set consisting of both of them.}} 
And in that respect black and white are indeed reasonably similar. 
And while they are poles apart along that dimension of comparison, grey is a mixture which is as close to white and to black as can be.  Moreover, because no property other than grey can be close to both white and black while being distinct from either, grey is even \emph{very close} to both.

Admittedly, a central grey, say of RGB value (128, 128, 128), may be seen as relatively distant from a central white (255, 255, 255) and from a central black (0, 0, 0). But in our analogy we need not equate ``grey'' with a specific triple of RGB values. Consider the region of triples of form $(x,x,x)$ with $0<x<255$. This is the grey region (including dark and light greys), and it is very similar to the (just) white region, and very similar to the (just) black region. Structurally, our point is that ``Both-True-and-False'' may be viewed in the same way, as denoting a region of overlap.\footnote{\addPE{This analogy can be made rigorous: \cite{egre2021half} treats ``true'' and ``false'' as \emph{absolute} gradable predicates, structurally ambiguous between a \emph{total} and a \emph{partial} interpretation (in the sense of \cite{rotstein2004total}). The total interpretations denote disconnected endpoints on the scale, but the partial interpretations overlap. In the case of achromatic colors, the total interpretation of ``white'' in RGB triples is conventionally $\{(255, 255, 255)\}$, but the partial one is the upset $\{(x,x,x) | k\leq x \leq 255\}$, and dually for ``black'', its total interpretation is conventionally $\{(0, 0, 0)\}$, and the partial one the downset $\{(x,x,x) | 0\leq x\leq k'\}$: provided $k<k'$, the partial interpretations overlap.}}

Statistical theories of vagueness naturally fit with this interpretation of the third value, but this interpretation too does not give rise to jolts (see \cite {borel2014economic,egre:pas,lassiter2011vagueness,eb:borel,lassiter2017adjectival, egre2017vague}).
On such an approach, we can view 1 as applying to an item for which the response can only be of the form ``$x$ is $P$'', $\half$ as encoding an item for which responses can be either ``$x$ is $P$'' or ``$x$ is not $P$'', and 0 as encoding an item for which the response can be only of the form ``$x$ is not $P$".
(Of course, items that get assigned the value $\half$ on this picture may still be such as to support different proportions or probabilities of ``$P$''-responses over ``not $P$''-responses. 
But that is not to say that one should necessarily semanticize those different proportions into distinct truth values.)
On this interpretation, the last item that gets the value 1 in a sorites and the first that gets the value $\half$ still have no jolt between them, in that they warrant identical responses on most occasions.

What the glutty and statistical views have in common is that the value $\half$ is seen as encoding a way in which the statuses represented by 1 and 0 \emph{interact}, rather than encoding a distinct status. This makes them altogether different from an interpretation on which $\half$ encodes a distinct response from the ones encoded by 1 or 0, such as ``I don't know'' or ``Indeterminate''. These third-status interpretations predict two $P$-similar (and so, consecutive in a sorites series) items that \emph{mandate} distinct responses, instead of \textit{allowing} for identical responses.
It is only these third-status interpretations that fall victim to Smith's jolt problem. (For related discussion, see \cite{wright:quandary, wright:vfca}.) Smith's argument, then, isn't fully about the \emph{number} of values in play, but instead turns also on their \emph{interpretation}.

Consequently, while we agree with Smith that a theory of vagueness working with only two values will incur the jolt problem, our point is that three truth values can suffice to give adequate provision against it. Like Smith, we also agree that the following version of Tolerance (stated p. 160), which we call Smith-tolerance, should be rejected:

\ex. If $a\sim_{P} b$ then $\val{Pa}=\val{Pb}$ \qquad (Smith-Tolerance)

But unlike Smith, we think that the tolerance principle in the version stated in \ref{tol} can be preserved as a first-order claim in the object language of a theory of vagueness (see \cite{cervr:tcs}). In what follows, we will show that a three-valued version of Closeness is actually enough to support Tolerance as an object-language principle; since it is this object-language formulation that generates the best form of the jolt problem, we conclude that three-valued approaches need not be subject to jolts.}

\section{Smith-consequence and ST: a comparison} \label{smithst}

In this section, we present and briefly compare two logical systems: the one adopted in \cite{smith:vdt}, which we will call `Smith-consequence', and a slightly modified version of the one adopted in \cite{cervr:tcs}, which we will call `ST'. (The purpose of the modifications is twofold: for simplicity, and to ease comparison with Smith's approach.) At first, we will consider versions of these logics that completely ignore the connection between the object-language predicates $P$ and $\sim_P$. As far as we are concerned for now, these are simply distinct predicates, and do not constrain each other in any way. We will introduce connections between these predicates later, once we have described the common structure that underlies both Smith-consequence and ST.

\subsection{Smith}

Smith's approach to vagueness is based on models that assign values from the real interval $[0, 1]$ to formulas; these models assign values to compound formulas compositionally along the usual \L ukasiewicz lines {(viz. \cite{luk1920}) for the 3-valued case)}, {\em except} for the conditional. Smith leaves out \L ukasiewicz's conditional entirely, instead defining a material conditional $A \to B$ to be equivalent to $\neg A \vee B$. In sum, then, Smith's models are pairs $\tuple{D, I}$ of a domain and an interpretation, where:
\begin{itemize}
%\item $I(A) \in [0, 1]$
\item For a term $t$, $I(t) \in D$
\item For an $n$-ary predicate $P$, $I(P) \in [0, 1]^{(D^n)}$
\item For an atomic sentence $A = P(t_1, \ldots, t_n), I(A) = I(P)(I(t_1), \ldots, I(t_n))$
\item $I(\neg A) = 1 - I(A)$
\item $I(A \wedge B) = \min(I(A), I(B))$
\item $I(A \vee B) = \max(I(A), I(B))$
\item $I(A \to B) = \max(1 - I(A), I(B))$
\item $I(\forall x A(x)) = \glb \{I'(A(x)) : I' \text{ is an $x$-variant of } I\}$
\item $I(\exists x A(x)) = \lub \{I'(A(x)) : I' \text{ is an $x$-variant of } I\}$
%\item $I(\forall x A(x)) = \glb\{I(A(t)) : t \text{ is a term}\}$
%\item $I(\exists x A(x)) = \lub\{I(A(t)) : t \text{ is a term}\}$
\end{itemize}

These are the models we will work with for the remainder of the paper; we will simply call them {\em models}. (When it is convenient, for a model $M = \tuple{D, I}$ we will sometimes write $M(A)$ instead of $I(A)$ for the value of a formula $A$.) 

An \emph{argument}, for our purposes here, is something of the form $\sqq{\G}{\Del}$, where $\G$ and $\Del$ are sets of formulas; this should be thought of as the argument with premises $\G$ and conclusions $\Del$.
A \emph{consequence relation} is a set $\vDash$ of arguments; we will write $\SQQ{\G}{\Del}$ for the claim that $\sqq{\G}{\Del} \in\; \vDash$.
To specify a consequence relation $\vDash$ model-theoretically, we simply specify a relation---the {\em countermodel} relation---between models and arguments; we then say that $\G \vDash \Del$ iff there is no model $M$ such that $M$ is a countermodel to $\sqq{\G}{\Del}$.

Although the models we work with are relatively familiar, Smith's approach to consequence is not a usual one for models like these.

\begin{defn} \label{smithdefn}
A model $\tuple{D, I}$ is a {\em Smith countermodel} to an argument $\sqq{\G}{\Del}$ iff: 
\begin{itemize}
\item $I(\gamma) > .5$ for every $\gamma \in \G$ and 
\item $I(\delta) < .5$ for every $\delta \in \Del$.
\end{itemize}
If there is no Smith countermodel to an argument $\sqq{\G}{\Del}$, then the argument is {\em Smith valid}, written $\G \Sc \Del$.\footnote{Smith considers single-conclusion arguments only; this is the most natural generalization of his approach to a multiple-conclusion setting, which helps to bring out the symmetry in the definition.}
\end{defn}

This approach effectively treats the models as having {\em three} inferential statuses to assign to formulas. A model can assign a formula a value strictly greater than .5, suitable to be the value of a premise of an argument in a countermodel to that argument; or a value strictly less than .5, suitable to be the value of a conclusion in a countermodel; or the value .5 itself, not suitable to be the value of a premise or a conclusion in a countermodel. 

Unlike usual designated-value or order-theoretic understandings of consequence on these models, understandings like Smith's that divide the values into {\em three} chunks, and require a countermodel to an argument to map its premises into one particular chunk and its conclusions to another, are not guaranteed to be {\em transitive} in any sense. Indeed, this is just the strategy exploited to produce nontransitivity in \cite{frankowski:fpi, zardini:mot, cervr:tcs}.\footnote{See also \cite{paris2009inconsistency} for a definition of consequence relative to two thresholds in a probabilistic setting, to deal with inconsistent beliefs. The resulting consequence is also called ``parameterized'' there; our terminology below was introduced independently.}

In fact, there is a sense in which \emph{every} consequence relation of a certain sort has a three-valued presentation: all that is required is that the consequence relation be \emph{monotonic} (such that adding premises or conclusions can never make an invalid argument out of a valid one) and \emph{reflexive} (such that every singleton set is a valid consequence of itself).
Nothing like transitivity is required.\footnote{\addEP{That \emph{four} values suffice for every \emph{monotonic} consequence relation, whether or not it is reflexive, is implicit in Proposition 2 of \cite[p.\ 407]{humberstone:hl}; the move to three values to impose reflexivity is given (again, implicitly) later on the same page. See also \cite{malinowski2004strawsonian, french2019valuations, blasio2017inferentially, chemla2019suszko} for different presentations of this fact. The representability of a monotonic logic by means of four values is connected to \emph{Suszko's Thesis} (\cite{suszko1977fregean}), which draws on a similar fact, imposing reflexivity and a strong form of transitivity to reduce the number of needed values to two. When only one of those two conditions is imposed, the number can be reduced to three (see \cite{malinowski:cilmv} about dropping reflexivity, \cite{frankowski:fpi} about dropping transitivity, and \cite{tsuji:suszko,blasio2017inferentially, french2019valuations, chemla2019suszko} for general results). %(For discussion and extensions of Suszko's Thesis, see for example \cite{blasio2017inferentially, french2019valuations, chemla2019suszko, font:tdts, malinowski:cilmv, tsuji:suszko}.)
Importantly, the fact that a consequence relation has \emph{some} three-valued presentation does nothing to show that it has a \emph{well-behaved} three-valued presentation.
In particular, there is no requirement of compositionality; it might be that the value of a compound sentence floats completely free of the values of its components.
In what follows, we {do} pursue some reductions to three-valued presentations, but these all maintain compositionality; the reductions we use are all truth-functional in the sense of \cite{chemla2019suszko}.}}
%and so are not guaranteed to exist simply by this Suszko-like fact.}

As it happens, despite the presence of three inferential statuses, the possibility of nontransitivity is not realized in \cite{smith:vdt}; Smith-consequence \emph{is} transitive. In fact, it is precisely the usual consequence relation of classical logic (\cite[p.\ 222]{smith:vdt}; this will also follow from our Theorem \ref{paraclassical} below).

%Comment on algebraization? Nothing Boolean here!

\subsection{ST}

The strict-tolerant approach to vagueness was first presented in \cite{cervr:tcs} as a(n implicitly) four-valued system.\footnote{In \cite{cervr:tcs}, there are three different satisfaction relations between models and formulas: tolerant, classical, and strict.
Each is implied by the one(s) that follow it.
As a result, there are four statuses a formula can have on a model: it can be strictly satisfied, classically but not strictly satisfied, tolerantly but not classically satisfied, or not satisfied at all.
The status of each compound sentence is determined by the statuses of its components.
This is the sense in which the system is four-valued.}
Here, we give a simpler three-valued formulation, following \cite{cervr:pive}.

\begin{defn} \label{stdefn}
A model $\tuple{D, I}$ is an {\em ST countermodel} to an argument $\sqq{\G}{\Del}$ iff:
\begin{itemize}
\item $I(A) \in \{0, \half, 1\}$ for every formula $A$; 
\item $I(\gamma) = 1$ for every $\gamma \in \G$; and
\item $I(\delta) = 0$ for every $\delta \in \Del$. 
\end{itemize}
If there is no ST countermodel to an argument $\sqq{\G}{\Del}$, then the argument is {\em ST valid}, written $\G \STc \Del$.
\end{defn}

Note that the first clause of this definition amounts to throwing out all models that use any values other than those in $\{0, .5, 1\}$. (If such models cannot be countermodels, then there is no reason to attend to them at all when evaluating an argument for ST validity.) The remaining three-valued models are a familiar sort: they are {\em strong Kleene} models (see e.g.\ \cite{bf:pp, priest:intro}). To present ST in its own right, strong Kleene models are the simplest (model-theoretic) tool, but here we work within the broader space of models, to preserve convenient links with other consequence relations.

Note that every ST countermodel to an argument is also a Smith countermodel to that argument; it follows immediately that $\G \Sc \Del$ implies $\G \STc \Del$. As it happens, the converse holds as well; both consequence relations are exactly the familiar consequence relation of first-order classical logic. (Again, this will follow from Theorem \ref{paraclassical}.)

\section{Parameterised consequence} \label{paracons}

In this section, we turn to a broad family of consequence relations, of which Smith-consequence and ST-consequence are two instances. A third instance will also be helpful:

\begin{defn} \label{classdefn}
A model $\tuple{D, I}$ is a {\em classical countermodel} to an argument $\sqq{\G}{\Del}$ iff:
\begin{itemize}
\item $I(A) \in \{0, 1\}$ for every formula $A$;
\item $I(\gamma) = 1$ for every $\gamma \in \G$; and
\item $I(\delta) = 0$ for every $\delta \in \Del$.
\end{itemize}
If there is no classical countermodel to an argument $\sqq{\G}{\Del}$, then the argument is {\em classically valid}, written $\G \Cc \Del$.
\end{defn}

Inspection of our models reveals that this is just the usual notion of a classical countermodel, and so the usual notion of (first-order) classical validity.

To move to the general case, we need the notion of a set of values being {\em closed}:

\begin{defn}
A set $V \subseteq [0, 1]$ is {\em closed} iff it is closed under greatest lower bound, least upper bound, and the function $\neg(x) = 1-x$.
\end{defn}

The import of this definition is contained in the following fact:
\begin{fact} \label{closedfact}
If $V \subseteq [0, 1]$ is closed, then for any model $\tuple{D, I}$: if for all $n$, for every $n$-ary predicate $P$, $I(P) \in V^{(D^n)}$, then for all formulas $A$, $I(A) \in V$.
\end{fact}

\begin{proof}
Induction on $A$'s formation. (Note that binary minimum and maximum are special cases of greatest lower bound and least upper bound, respectively.)
\end{proof}

That is, for closed $V$, if a model assigns predicate values built only from values in $V$, then the entire model will assign values only from $V$. This notion is convenient for identifying usable selections from the value space $[0, 1]$. Note that every finite subset of $[0, 1]$ that is closed under the function $\neg(x) = 1 -x$ is closed, since every finite subset is closed under greatest lower bound and least upper bound {(of sets)}. (In finite cases, these simply amount to minimum and maximum, respectively.) In particular, both $\{0, \half, 1\}$ and $\{0, 1\}$ are closed.

We consider a parameterised notion of countermodel. Our parameters have three coordinates: a set $V$ of values; a set $T$ of values for premises to take in a countermodel; and a set $F$ of values for conclusions to take in a countermodel. Given these parameters, a countermodel is a model that takes values from $V$, maps all premises into $T$, and maps all conclusions into $F$. We will not consider all possible ways of doing this, but we will still consider quite a wide range.

\begin{defn}
A set $X \subseteq [0, 1]$ is an {\em upset} iff whenever $x \in X$ and $x < y \leq 1$, then $y \in X$; it is a {\em downset} iff whenever $x \in X$ and $x > y \geq 0$, then $y \in X$.
\end{defn}

\begin{defn}
A {\em parameter} is a triple $\tuple{V, T, F}$ such that: 
\begin{itemize}
\item $\{0, 1\} \subseteq V \subseteq [0, 1]$ is closed;
%\item $V \subseteq [0, 1]$ is closed;
\item $1 \in T \subseteq (.5, 1]$ is an upset; and
%\item $T \subseteq V$ is a nonempty upset; and
\item $0 \in F \subseteq [0, .5)$ is a downset.
%\item $F \subseteq V$ is a nonempty downset.
\end{itemize}
\end{defn}

\begin{defn} \label{paramdefn}
Given a parameter $\para = \tuple{V, T, F}$, a model $\tuple{D, I}$ is a {\em $\para$ countermodel} to an argument $\sqq{\G}{\Del}$ iff:
\begin{itemize}
\item $I(A) \in V$ for every formula $A$;
\item $I(\gamma) \in T$ for every $\gamma \in \G$; and
\item $I(\delta) \in F$ for every $\delta \in \Del$.
\end{itemize}
If there is no $\para$ countermodel to an argument $\sqq{\G}{\Del}$, then the argument is {\em $\para$ valid}, written $\G \pcp \Del$.
\end{defn}

Note that all three of our examples fit this mould: for Smith-consequence, the parameter is $\tuple{[0, 1], (.5, 1], [0, .5)}$; for ST-consequence, $\tuple{\{0, \half, 1\}, \{1\}, \{0\}}$; and for classical consequence, $\tuple{\{0, 1\}, \{1\}, \{0\}}$. By letting S, ST, and C simply {\em be} these parameters, we can see Definitions \ref{smithdefn}, \ref{stdefn} and \ref{classdefn} all as special cases of Definition \ref{paramdefn}.{\footnote{{Our definition of a parameter rules out the overlap between $T$ and $F$. By allowing overlap, and taking $T\subseteq (0,1]$, and $F\subseteq [0,1)$, we could retrieve another notion of entailment explored in the literature, namely TS entailment (see \cite{cervr:tcs}), also known as Q-consequence (\cite{malinowski1990q, malinowski2009beyond}), then expressible as $\tuple{\{0, \half, 1\}, \{\half, 1\}, \{0, \half\}}$. We leave an exploration of TS entailment for another occasion.}}}

We proceed to characterize $\para$ validity for an arbitrary parameter, by way of a definition and a pair of lemmas.%\footnote{\addEP{What Definition \label{crispdefn} operates is basically a Suszko-reduction, viz. \cite[p. 204-205]{malinowski2009beyond}.}}

\begin{defn} \label{crispdefn}
A model $M' = \tuple{D, I'}$ {\em crispifies} a model $M = \tuple{D, I}$ (written $M \iscrisped M'$) iff for all $n$, for all $n$-ary predicates $P$:
\begin{itemize}
\item $I'(P) \in \{0, 1\}^{(D^n)}$;
\item if $I(P)(\tuple{d_1, \ldots, d_n}) > .5$, then $I'(P)(\tuple{d_1, \ldots, d_n}) = 1$; and
\item if $I(P)(\tuple{d_1, \ldots, d_n}) < .5$, then $I'(P)(\tuple{d_1, \ldots, d_n}) = 0$.
\end{itemize}
\end{defn}

Note that when $M \iscrisped M'$, then $M'$ assigns predicates only values from $\{0, 1\}$; since this set is closed, we already know that $M'$ thus assigns every formula some value in $\{0, 1\}$. But something more interesting is happening here as well:

\begin{lem} \label{crispfact}
If $M \iscrisped M'$, then for every formula $A$, if $M(A) > .5$ then $M'(A) = 1$, and if $M(A) < .5$ then $M'(A) = 0$.
\end{lem}

\begin{proof}
Induction on $A$'s formation.
\end{proof}

\begin{lem} \label{hascrisp}
Every model can be crispified. That is, for any $M$, there is some $M'$ such that $M \iscrisped M'$.
\end{lem}

\begin{proof}
Definition \ref{crispdefn} already tells us how to crispify each case where $I(P)$ assigns a value different from .5; for cases where $I(P)$ assigns exactly .5, either 1 or 0 will do.
\end{proof}

We are now ready to characterize $\para$ validity:

\begin{thm} \label{paraclassical}
For any parameter $\para$, $\G \pcp \Del$ iff $\G \Cc \Del$.
\end{thm}

\begin{proof}
For each direction, we show the contrapositive. 

LTR: suppose $\G \not\Cc \Del$. Then there is some classical countermodel $M$ to $\sqq{\G}{\Del}$, some $M$ such that $M(A) \in \{0, 1\}$ for every formula $A$, $M(\gamma) = 1$ for every $\gamma \in \G$ and $M(\delta) = 0$ for every $\delta \in \Del$. But since $\para = \tuple{V, T, F}$ is a parameter, this gives $M(A) \in V$ for every formula $A$ (since $\{0, 1\} \subseteq V$), $M(\gamma) \in T$ for every $\gamma \in \G$ (since $1 \in T$), and $M(\delta) \in F$ for every $\delta \in \Del$ (since $0 \in F$). Thus, $M$ is a $\para$ countermodel to $\sqq{\G}{\Del}$, and so $\G \not\pcp \Del$.

RTL: suppose $\G \not\pcp \Del$. Then there is some $\para$ countermodel $M$ to $\sqq{\G}{\Del}$. This requires that $M(\gamma) \in T$ for every $\gamma \in \G$ and $M(\delta) \in F$ for every $\delta \in \Del$. Since $\para$ is a parameter, this gives $M(\gamma) > .5$ for every $\gamma \in \G$ and $M(\delta) < .5$ for every $\delta \in \Del$. By Lemma \ref{hascrisp}, there is some model $M'$ such that $M \iscrisped M'$. By Lemma \ref{crispfact}, $M'(\gamma) = 1$ for every $\gamma \in \G$ and $M'(\delta) = 0$ for every $\delta \in \Del$. Since $M'(A) \in \{0, 1\}$ for every formula $A$, $M'$ is a classical countermodel to $\sqq{\G}{\Del}$, and so $\G \not\Cc \Del$.
\end{proof}

From this perspective, it is no surprise that Smith-consequence and ST-consequence both turn out to be exactly classical logic; these are just two pinpricks of light shed on the broader phenomenon here, which is that {\em every} parameterised consequence relation is exactly classical logic.

This immediately yields an $n$-valued presentation of classical logic for every $n \geq 2$. 
Let $V_n = \{0, 1/(n - 1), \ldots , (n-2)/(n - 1), 1\}$. 
Then $V_n$ has $n$ members, and $\para = \tuple{V_n, \{1\}, \{0\}}$ is a parameter. It also reveals that Smith's choice of the parameter $\tuple{[0, 1], (.5, 1], [0, .5)}$ is logically arbitrary, even given his choice of the value space $[0, 1]$.
Many choices for the parameter's last two coordinates would have yielded the same consequence relation.

Smith offers more motivation for his choice than simply the consequence relation it yields.
He understands a sentence with value $\geq \half$ as `assertion grade' (fit to assert) and a sentence with value $> \half$ as `inference grade' (fit to infer from); the idea is that a valid argument whose premises are all inference grade must have some conclusion that is assertion grade.
Why these particular values?
\begin{quote}
The advantage of my proposal\ldots is that it is \emph{minimal}. 
If a sentence $S$ is at least 0.5 true, then one cannot make a truer statement by asserting the negation of $S$ than by asserting $S$.
What more than this could be required for a statement to be `assertion grade'\ldots?
Any higher standard would need further justification, and I cannot see what such justification would consist in.
Now, given that we have set the cut-off for assertion grade statements at 0.5, and want to make the cut-off for inference grade statements strictly \emph{higher} than this, the \emph{minimal} cut-off for inference grade statements will be the one I have proposed: they must be \emph{more than} 0.5 true.
Again, any higher standard would need further justification, and I cannot see what such justification would consist in. \cite[p.\ 224, emphases in original]{smith:vdt}
\end{quote}

We don't see, however, that minimality in this sense is any advantage to a proposal at all.
Exactly what benefit is having smaller numbers supposed to confer on a theory?
Smith offers no answer.\footnote{The above-quoted passage is the full discussion of the issue in \cite{smith:vdt}, except for (p.\ 250, fn.\ 57, emphasis in original): %, which, if taken at face value, seems to render Smith's above-quoted account of `assertion grade' circular (and so trivially true, rather than justified by any appeal to minimality)
`[Earlier], I said that a sentence is `assertion grade'\ldots if its degree of truth is greater than or equal to 0.5.
This does \emph{not} mean that if a sentence $S$ has a degree of truth of 0.5 or greater, then an \ldots assertion of $S$ is acceptable.
Rather, the idea is that a sentence is `assertion grade'\ldots if the level of confidence appropriate in an utterance of the sentence is at least as high as the level of confidence appropriate in an utterance of its negation.'
We set aside further discussion of how exactly to understand `assertion grade' and `inference grade' on Smith's account.
}

We conclude, then, that for capturing the logical behaviour common to ST and to Smith's approach before we take account of similarity predicates, any parameter will do as well as any other.
All yield the same logic (ordinary classical logic), and we see no other potential reason to choose between them (so long as the values are interpreted in a way that circumvents the jolt problem, as discussed in \S\ref{jolt}).

\subsection{Tolerant logics} \label{tolerance}

Here, we move on to consider the logic of $P$-similarity, registering the connections between $\sim_P$ and $P$. For tolerant logics, we impose a connection between the values assigned to $t \sim_P u$, to $Pt$, and to $Pu$, for every predicate $P$ and terms $t, u$. It is perhaps not immediately apparent how to understand these connections with regard to Smith-consequence, so we pursue an indirect approach. First we consider the situation as developed in \cite{cervr:tcs, cervr:pive}. Then we turn to the general case, first looking at parameterised consequence relations in their full generality, and then narrowing in on a particular class of them; ST-consequence will be seen as a member of this wider class.

We begin by narrowing our space of models to ensure that $\sim$ relations are one and all reflexive and symmetric, in the following sense: for all predicates $P$ and all terms $t, u$: $I(t \sim_P t) = 1$ and $I(t \sim_P u) = I(u \sim_P t)$. From here forward, we ignore models that do not obey these restrictions.\footnote{This of course
already results in more valid arguments than we already had.
Look ahead to Figure \ref{stsimcalc}. By removing the rule Tol from the
sequent calculus there, you arrive at a calculus sound and
complete for every $\para$ consequence relation obeying these $\sim$
restrictions. (Yes, they are all the same; the argument is
the same as for Theorem \ref{paraclassical}, mutatis mutandis).}

%(This of course already results in more valid arguments than we already had, but we do not consider consequence relations that impose only these restrictions; we're after bigger game.\footnote{Okay, fine: look ahead to Figure \ref{stsimcalc}. By removing the rule Tol from the sequent calculus there, you arrive at a calculus sound and complete for every $\para$ consequence relation obeying these $\sim$ restrictions. (Yes, they're all the same; the argument is the same as for Theorem \ref{paraclassical}, mutatis mutandis.)})

\subsubsection{\STsim} \label{fullst}

\begin{defn}
A model $M$ obeys the {\em \STsim\ restriction} iff for all predicates $P$ and terms $t, u$: if $M(Pt) = 1$ and $M(Pu) = 0$, then $M(t \sim_P u) = 0$.
\end{defn}

Intuitively, the \STsim\ restriction ensures that if $t$ and $u$ are so $P$-unalike as to go all the way from 1 to 0 in their $P$-value, then they must not be at all $P$-similar.

\begin{defn} \label{sttolerant}
A model $M$ is an {\em \STsim\ countermodel} to an argument $\sqq{\G}{\Del}$ iff:
\begin{itemize}
\item $M$ obeys the \STsim\ restriction, and 
\item $M$ is an ST counterexample to $\sqq{\G}{\Del}$.
\end{itemize}
If there is no \STsim\ countermodel to an argument $\sqq{\G}{\Del}$, then the argument is {\em \STsim\ valid}, written $\G \STsc \Del$. 
\end{defn}

It will be convenient later to have a proof system for \STsim. We will work with a sequent calculus, given in Figure \ref{stsimcalc}. In the figure, $t$ and $u$ can be any terms, and $a$ must be an {\em eigenvariable}: a variable that does not occur free in the conclusion-sequent of the rule.

\begin{figure}
\vskip 3mm
\def\fCenter{\, \Yright \,}
\begin{center}
{\bf Structural rules:}
\vs
\AXC{\phantom{$A$}}
\LLl{Id}
\UIx{A}{A}
\DP
\hs
\AXx{\G}{\Del}
\LLl{K}
\UIx{\G, \G'}{\Del, \Del'}
\DP

\vs
{\bf Operational rules:}
\vs

\AXx{\G}{A, \Del}
\LLl{$\neg$L}
\UIx{\G, \neg A}{\Del}
\DP
\hs
\AXx{\G, A}{\Del}
\LLl{$\neg$R}
\UIx{\G}{\neg A, \Del}
\DP
\vs
\AXx{\G, A, B}{\Del}
\LLl{$\wedge$L}
\UIx{\G, A \wedge B}{\Del}
\DP
\hs
\AXx{\G}{A, \Del}
\AXx{\G}{B, \Del}
\LLl{$\wedge$R}
\BIx{\G}{A \wedge B, \Del}
\DP
\vs
\AXx{\G, A}{\Del}
\AXx{\G, B}{\Del}
\LLl{$\vee$L}
\BIx{\G, A \vee B}{\Del}
\DP
\hs
\AXx{\G}{A, B, \Del}
\LLl{$\vee$L}
\UIx{\G}{A \vee B, \Del}
\DP
\vs
\AXx{\G}{A, \Del}
\AXx{\G, B}{\Del}
\LLl{$\to$L}
\BIx{\G, A \to B}{\Del}
\DP
\hs
\AXx{\G, A}{B, \Del}
\LLl{$\to$R}
\UIx{\G}{A \to B, \Del}
\DP
\vs
\AXx{\G, A(t)}{\Del}
\LLl{$\forall$L}
\UIx{\G, \forall x A(x)}{\Del}
\DP
\hs
\AXx{\G}{A(a), \Del}
\LLl{$\forall$R}
\UIx{\G}{\forall x A(x), \Del}
\DP
\vs
\AXx{\G, A(a)}{\Del}
\LLl{$\exists$L}
\UIx{\G, \exists x A(x)}{\Del}
\DP
\hs
\AXx{\G}{A(t), \Del}
\LLl{$\exists$R}
\UIx{\G}{\exists x A(x), \Del}
\DP

\vs
{\bf Similarity rules:}
\vs

\AXx{\G, t \sim_P t}{\Del}
\LLl{$\sim$ref}
\UIx{\G}{\Del}
\DP
\vs
\AXx{\G, t \sim_P u}{\Del}
\LLl{$\sim$symL}
\UIx{\G, u \sim_P t}{\Del}
\DP
\hs
\AXx{\G}{t \sim_P u, \Del}
\LLl{$\sim$symR}
\UIx{\G}{u \sim_P t, \Del}
\DP
\vs
\AXx{\G}{t \sim_P u, \Del}
\LLl{Tol}
\UIx{\G, Pt}{Pu, \Del}
\DP

\end{center}

\caption{A sequent calculus for \STsim}
\label{stsimcalc}
\end{figure}

\begin{fact} \label{stsc}
The sequent calculus in Figure \ref{stsimcalc} is sound and complete for \STsim.
\end{fact}

\begin{proof}
Both facts are straightforward to show in the usual ways. 
In particular, completeness can be proved following the method of \cite{takeuti:pt}, which builds a countermodel from an underivable sequent.
For the connectives and quantifiers, this method is adapted to strong Kleene models in \cite{ripley:pafc}.

The only needed addition is to show that the resulting model handles $\sim$ appropriately, but this is ensured by the $\sim$-involving rules.
In particular, $\sim$ref ensures that the resulting model assigns 1 to all formulas of the form $t \sim_P t$, the $\sim$sym rules ensure that the resulting model assigns the same value to $t \sim_P u$ as to $u \sim_P t$, and the Tol rule guarantees that the resulting model obeys the \STsim\ restriction.
\end{proof}

\STsim\ is an {\em expansion} of first-order classical logic. 
\STsim\ countermodels are all ST countermodels, so ST validity (which we know is classical validity from Theorem~\ref{paraclassical}) implies \STsim\ validity. \STsim\ also has the nice feature that it validates the principle of tolerance. This is so whether we consider instances of tolerance as {\em arguments} (since $Pa, a \sim_P b \STsc Pb$) or as {\em quantified conditionals} (since $\STsc \forall x \forall y ((Px \wedge x \sim_P y) \to Py)$). The main form of tolerance we're interested in, though, is the {\em metainferential} form, expressed by the rule Tol in Figure \ref{stsimcalc}.

%When we say that a consequence relation is {\em tolerant}, we mean that it is closed under this rule. 

\begin{defn}
A consequence relation is {\em tolerant} iff it is closed under the rule Tol.
\end{defn}

\noindent Since Tol is part of a sound and complete sequent calculus for \STsim, \STsim\ is tolerant. In this calculus, the other forms of tolerance follow from the metainferential form, via the rules Id, $\wedge$L, $\to$R, and $\forall$R. We will return to this in \S\ref{metainferences}.

There is a natural worry at this point: doesn't \STsim\ fall right into the sorites paradox? After all, quantified conditional tolerance plus classical logic is typically thought to be enough for trouble, yet \STsim\ validates these. But \STsim\ has an escape route: the structure of its countermodels allows it to be {\em nontransitive}. Indeed, \STsim\ {\em is} nontransitive, and this allows it to escape the looming trouble. We do not pursue the ups and downs of nontransitivity here; \cite{cervr:tcs, cervr:pive} offer related discussion.

\subsubsection{Parameter tolerance} \label{parameter} \label{fullgeneral}

Our concern here is to explore various avenues for restricting parameterised consequence so as to respect the connection between $P$ and $\sim_P$. We take as our paradigm the \STsim\ restriction, and generalise it to what we call {\em parameter tolerance}.

In any parameter $\para = \tuple{V, T, F}$, $T$ is some sort of positive status, and $F$ some kind of negative status. The force of $\para$ validity is to guarantee that if the premises all have the positive status, then some conclusion lacks the negative status. Parameter tolerance is parasitic on these statuses: it guarantees that if $Pt$ has the positive status while $Pu$ has the negative one, then $t \sim_P u$ must also have the negative one.

\begin{defn}
For a parameter $\para = \tuple{V, T, F}$, a model $M$ is {\em $\para$ tolerant} iff for all predicates $P$, for all terms $t, u$: if $M(Pt) \in T$ and $M(Pu) \in F$, then $M(t \sim_P u) \in F$.
\end{defn}

We can use parameter tolerance to get a new range of parameterised consequence relations:

\begin{defn} \label{paratolerant}
A model $M$ is a {\em $\para$ tolerant countermodel} to an argument $\sqq{\G}{\Del}$ iff:
\begin{itemize}
\item $M$ is $\para$ tolerant, and
\item $M$ is a $\para$ countermodel to $\sqq{\G}{\Del}$.
\end{itemize}
If there is no $\para$ tolerant countermodel to an argument $\sqq{\G}{\Del}$, then the argument is {\em $\para$ tolerant valid}, written $\G \pscp \Del$.
\end{defn}

Note that with ST understood as the parameter $\tuple{\{0, \half, 1\}, \{1\}, \{0\}}$, as before, Definition \ref{paratolerant} subsumes Definition \ref{sttolerant}. Parameter tolerance brings with it metainferential tolerance.

\begin{fact} \label{paratol}
For every parameter $\para$, $\pscp$ is tolerant.
\end{fact}

\begin{proof}
Let $\para = \tuple{V, T, F}$, and suppose that $\G\ \cup\ \{Pt\} \not\pscp \{Pu\}\ \cup\ \Del$. Then there is a $\para$ tolerant model $M$ such that $M(\gamma) \in T$ for every $\gamma \in \G \cup \{Pt\}$ and $M(\delta) \in F$ for every $\delta \in \Del \cup \{Pu\}$. 
In particular, $M(Pt) \in T$ and $M(Pu) \in F$. Since $M$ is $\para$ tolerant, then, $M(t \sim_P u) \in F$. But then $\G \not\pscp t \sim_P u, \Del$, since $M$ is a countermodel to this argument as well.
\end{proof}

As the proof reveals, $\para$ tolerance is just what is needed to guarantee metainferential tolerance. Since all parameterised consequence relations are reflexive, this is enough to guarantee argument-form tolerance as well. (For other forms of tolerance, we need more restrictions yet; leave those to one side for now.)

Moreover, these consequence relations are tied quite tightly to classical logic, as Fact \ref{paraconservative} records:

\begin{fact} \label{paraconservative}
For every parameter $\para$, $\pscp$ is a {\em conservative extension} of classical logic: if $\G \pscp \Del$ and $\G \not\Cc \Del$, then some $\sim$ predicate occurs in $\G \cup \Del$.
\end{fact}

\begin{proof}
Suppose that $\G \cup \Del$ contains no occurrences of any $\sim_P$ relation, and suppose that $\G \not\Cc \Del$. Then there is a classical countermodel $M$ for the language without $\sim$ relations. Extend $M$ to a model $M'$ of the full language as follows: for all terms $t, u$ and predicates $P$, if $M(t) = M(u)$, then $M'(t \sim_P u) = 1$; else $M'(t \sim_P u) = 0$. $M'$ is $\para$ tolerant, as is quick to check. But $M'(A) = M(A)$ for all $A \in \G \cup \Del$. So $M'$ is a $\para$ tolerant countermodel: $\G \not\pscp \Del$.
\end{proof}

The $\para$ tolerant consequence relations, then, are a way of extending classical logic to take account of the special behaviour of $\sim$ relations, without messing with anything that doesn't involve these relations.

\subsection{Sorites} \label{sorites}

As we flagged above, \STsim\ handles the combination of classical logic and tolerance without running into sorites problems via {\em nontransitivity}. A natural question that arises at this point, then, is: which of the $\pscp$ relations work the same way? The answer is: the {\em proper} ones.

\begin{defn}
A parameter $\para = \tuple{V, T, F}$ is {\em proper} iff $V \not\subseteq T \cup F$. The consequence relation $\pscp$ is {\em proper} iff $\para$ is.
\end{defn}

Proper parameters are parameters with some value that is neither in $T$ nor $F$; such a value is needed for any counterexample to transitivity to arise. With this notion in hand, we have the following containment result:

\begin{fact} \label{parainst}
For every proper parameter $\para$, $\pscp \;\subseteq\; \STsc$.
\end{fact}

\begin{proof}
Let $\para = \tuple{V, T, F}$, and suppose $\G \not\STsc \Del$. Then there is a model $M$ meeting the \STsim\ restriction such that: $M(\gamma) = 1$ for all $\gamma \in \G$, and $M(\delta) = 0$ for all $\delta \in \Del$. Since $V \not\subseteq T \cup F$, there must be some value in $V$ that is in neither $T$ nor $F$; call it $x$. Since $V$ is closed, $1-x \in V$ as well. Now consider the model $M'$ defined as follows: for all $n$ for all $n$-ary predicates $P$,
\begin{itemize}
\item If $M(P)(\tuple{d_1, \ldots, d_n}) = 1$, then $M'(P)(\tuple{d_1, \ldots, d_n}) = 1$;
\item if $M(P)(\tuple{d_1, \ldots, d_n}) = 0$, then $M'(P)(\tuple{d_1, \ldots, d_n}) = 0$; and
\item if $M(P)(\tuple{d_1, \ldots, d_n}) = .5$, then $M'(P)(\tuple{d_1, \ldots, d_n}) = x$.
\end{itemize}
Note that $\{0, x, 1-x, 1\}$ is closed, so $M'$ assigns only these values. Moreover, we can show by induction on $A$'s formation that for all formulas $A$, if $M(A) = 1$ or $0$, then $M'(A) = M(A)$. Thus, $M'$ still assigns 1 to everything in $\G$ and 0 to everything in $\Del$. Moreover, it assigns values only from $V$, so it is a $\para$ countermodel. It remains only to show that $M'$ is $\para$ tolerant. Suppose $M'(Pt) \in T$ and $M'(Pu) \in F$. There are three cases:
\begin{itemize}
\item If $M(Pt) = 1$ and $M(Pu) = 0$, we have $M(t \sim_P u) = 0$, since $M$ meets the \STsim\ restriction, and so $M'(t \sim_P u) = 0 \in F$. 
\item If $M(Pt) \neq 1$, then $M(Pt) = .5$ or $0$, and so $M'(Pt) = x$ or $0$. But either way, $M'(Pt) \not \in T$; contradiction.
\item If $M(Pu) \neq 0$, then $M(Pu) = .5$ or $1$, and so $M'(Pu) = x$ or $1$. But either way, $M'(Pu) \not\in F$; contradiction.
\end{itemize}
$M'$ is thus a $\para$ tolerant countermodel to $\sqq{\G}{\Del}$, and so $\G \not\pscp \Del$.
\end{proof}

The sequent calculus in Figure \ref{stsimcalc} is thus complete for every proper $\pscp$. 
(It is not, however, sound for all of them; we return to this in \S\ref{metainferences}.)

Fact \ref{parainst} does not hold for improper parameters. This is because improper parameters walk straight into the business end of the sorites paradox.

\begin{fact}
For any improper $\para$ and any $n$ terms $t_1, \ldots, t_n$: $Pt_1, t_1 \sim_P t_2, t_2 \sim_P t_3, \ldots, t_{n-1} \sim_P t_n \pscp Pt_n$.
\end{fact}

\begin{proof}
Let $\para = \tuple{V, T, F}$, and suppose the claim fails. Then there is some $\para$ tolerant countermodel $M$. It must be that $M(Pt_1) \in T$ \addEP{and that $M(Pt_n) \in F$}, hence $M(Pt_n) \not\in T$; so there is a first $i \leq n$ such that $M(Pt_i) \not\in T$. Since $i$ is the first such, $M(Pt_{i-1}) \in T$, and since $\para$ is improper, $M(Pt_i) \in F$. Since $M$ is $\para$ tolerant, then, $M(t_{i-1} \sim_P t_i) \in F$. But then $M$ cannot be a countermodel to this argument, since $t_{i-1} \sim_P t_i$ is among its premises. Contradiction.
\end{proof}

So while improper parameters give rise to tolerant conservative extensions of classical logic, they are not useful for exploring sorites sequences, as they fall victim to the sorites paradox, by requiring that if the first member of a sorites sequence for $P$ is $P$, then so is the last. Much of what follows holds for parameters whether or not they are proper, so we do not restrict our attention only to proper parameters, but we do think that improper parameters are unlikely to be of any help in treating vagueness, since they do not avoid the key problem posed by sorites sequences.

On the other hand, proper parameters give rise to nontransitivity just where it is needed. Consider the form of transitivity embodied in the rule of {\em Cut}:
\begin{prooftree}
\AXx{\G}{A, \Del}
\AXx{\G', A}{\Del'}
\LLl{Cut}
\BIx{\G, \G'}{\Del, \Del'}
\end{prooftree}

\begin{fact}
For any proper $\para$, $\pscp$ is not closed under Cut.
\end{fact}

\begin{proof}
Since every $\pscp$ is reflexive and tolerant, they all validate the arguments $\sqq{Pt_1, t_1 \sim_P t_2}{Pt_2}$ and $\sqq{Pt_2, t_2 \sim_P t_3}{Pt_3}$. If we could apply cut to these, we would reach $\sqq{Pt_1, t_1 \sim_P t_2, t_2 \sim_P t_3}{Pt_3}$. But this last is not \STsim\ valid, so not $\para$ tolerant valid for any proper $\para$ by Fact \ref{parainst}.
\end{proof}

Proper $\pscp$ relations thus handle sorites reasoning just like $\STsc$: by validating each step of the reasoning, but refusing to allow them to be chained together via Cut.

\subsection{Metainferences} \label{metainferences}

Theorem \ref{paraclassical} guarantees that every $\para$ tolerant consequence relation is {\em argument classical}: they all validate every instance of every classically-valid argument. But this is just one way to be classical; we might (and probably should) want more.

As we pointed out above, the sequent calculus in Figure \ref{stsimcalc} is complete for all proper parameter tolerant consequence, but it is not {\em sound} for all of them. This may at first seem counterintuitive, since every $\pscp$ strengthens classical logic and obeys the similarity rules, while these rules are the only rules in Figure \ref{stsimcalc} \emph{not} sound for classical logic. But counterintuitive or not, it is the case; we pause here to sort this out.

We can use the calculus of Figure \ref{stsimcalc} to identify certain {\em metainferences} important to classicality. (A {\em metainference} is a property that a consequence relation may or may not be closed under. For discussion, see e.g.\ \cite{field:stp, scharp:rt, cervr:rtt, brt:lstl}.) 

\begin{fact} \label{clutter} 
Every $\para$ tolerant consequence relation is closed under the metainferences given by the rules Id, K, $\sim$ref, $\sim$symL, and $\sim$symR.
\end{fact}

\begin{proof}
Straightforward. (For the $\sim$ rules, recall that we have restricted our models so that $M(t \sim_P t) = 1$ and $M(t \sim_P u) = M(u \sim_P t)$.)
\end{proof}

Facts \ref{paratol} and \ref{clutter} ensure that every $\para$ tolerant consequence relation obeys the structural rules and similarity rules given in Figure \ref{stsimcalc}. This leaves the operational rules; these determine what we will call the {\em operational metainferences}. We say that a consequence relation is {\em operationally classical} iff it is closed under all of these operational metainferences.

It is entirely possible for a consequence relation to be argument classical {\em without} being operationally classical. For example, consider the metainference determined by the rule $\to$R, and consider the parameter $\para = \tuple{[0, 1], \{1\}, [0, .5)}$. We have $Pa \wedge a \sim_P b \pscp Pb$, but $\not\pscp (Pa \wedge a \sim_P b) \to Pb$; for a countermodel to the latter, let $M(Pa) = M(a \sim_P b) = .6$ and $M(Pb) = .4$. Note that this model can be $\para$ tolerant, since $M(Pa) \not\in T$.

Here, then, we describe the situation for the operational metainferences, identifying sufficient conditions for a parameter $\para$ to yield an operationally classical $\para$ tolerant consequence relation.

\begin{fact} \label{parami}
Every parameterised consequence is closed under $\wedge$L, $\wedge$R, $\vee$L, $\vee$R, $\forall$L, and $\exists$R.
\end{fact}

\begin{proof}
Let $\para = \tuple{V, T, F}$.

For $\wedge$L: A countermodel to the conclusion-sequent must assign some value in $T$ to $A \wedge B$; but since this value is the minimum of the values of $A$ and $B$, and since $T$ is an upset, this model must assign some value in $T$ to each of $A$ and $B$, and so be a countermodel to the premise-sequent.

For $\wedge$R: A countermodel to the conclusion-sequent must assign some value in $F$ to $A \wedge B$; but since this value is the minimum of the values of $A$ and $B$, this model must assign that value (and so some value in $F$) to at least one of $A$ or $B$, and so be a countermodel to at least one premise-sequent.

For $\forall$L: A countermodel $M$ to the conclusion-sequent must assign some value in $T$ to $\forall x A(x)$; but since this value is a lower bound for the values of $A(x)$ in $x$-variants of $M$, and since $T$ is an upset, all of these $x$-variants must assign some value in $T$ to $A(x)$. There must be some $x$-variant $M'$ such that $M'(x) = M(t)$, so $M(A(t)) \in T$; $M$ is thus a countermodel to the premise-sequent.

$\vee$L is similar to $\wedge$R; $\vee$R to $\wedge$L; and $\exists$R to $\forall$L.
\end{proof}

This leaves two kinds of metainferences: the {\em negative} ones $\neg$L, $\neg$R, $\to$L, and $\to$R; and the {\em eigenvariable} ones $\forall$R and $\exists$L. Parameterised tolerant consequence relations as such are not guaranteed to be closed under any of these. The trouble with the negative ones is that there is no connection between $x \in T$ and $1-x \in F$; and the trouble with the eigenvariable ones is that the lub or glb of a set $X$ might be in $T$ or $F$ without any member of $X$ being so. To ensure that the remaining operational metainferences work, we need to tighten up our parameters in these two ways.

\begin{defn}
A parameter $\para = \tuple{V, T, F}$ is {\em symmetric} iff for all $x \in [0, 1]$: $x \in T$ iff $1 - x \in F$. The consequence relation $\pscp$ is {\em symmetric} iff $\para$ is.
\end{defn}

In a symmetric parameter $\tuple{V, T, F}$, $T$ and $F$ are mirror images of each other, with the function $\neg(x) = 1 - x$ serving as the mirror.

\begin{fact} \label{negmi}
$\para$ is symmetric iff $\pcp$ is closed under $\neg$L and $\neg$R iff $\pcp$ is closed under $\to$L and $\to$R.

%$\para$ is symmetric \addPE{iff} $\pscp$ is closed under the metainferences $\neg$L, $\neg$R, $\to$L, and $\to$R.
\end{fact}

\begin{proof}

{LTR:} Let $\para = \tuple{V, T, F}$.
For $\neg$L: A countermodel $M$ to the conclusion-sequent must have $M(\neg A) \in T$. Since it is symmetric, this gives $1 - M(\neg A) \in F$. But $M(A) = 1 - M(\neg A)$, so $M(A) \in F$, and $M$ is a countermodel to the premise-sequent.

$\neg$R is similar to $\neg$L. 

Since $M(A \to B) = M(\neg A \vee B)$ in every model, it suffices for $\to$L and $\to$R to show that $\neg A \vee B$ obeys these rules. But this can be derived from $\neg$L, $\neg$R, $\vee$L, and $\vee$R as follows:

\begin{center}
\AXx{\G}{A, \Del}
\LLl{$\neg$L}
\UIx{\G, \neg A}{\Del}
\AXx{\G, B}{\Del}
\LLl{$\to$L}
\BIx{\G, \neg A \vee B}{\Del}
\DP
\hs
\AXx{\G, A}{B, \Del}
\LLl{$\neg$R}
\UIx{\G}{\neg A, B, \Del}
\LLl{$\vee$R}
\UIx{\G}{\neg A \vee B, \Del}
\DP

\end{center}

%\addPE{(RTL) Suppose $\para$ is not symmetric. Then we can show that $\pscp$ is not closed under at least one of the four metainferences. If $x\notin T$ but $1-x\in F$, then one may have: $x\pscp 1-x$, but $\not \pscp 1-x, 1-x$ (i.e $\neg R$ and $\rightarrow R$ can fail) (example: let $T=[0.7,1]$, and $F=[0, 0.4]$, pick $x=0.65$). If $x\in T$ but $1-x\notin F$, then one may have: $\pscp 1-x$ and $1-x\pscp$ without $1-x \to 1-x \pscp$, and likewise one can have $\pscp 1-x$ without $x\pscp$   (i.e. $\to L$ and $\neg L$ can fail) (example: $T=[0.8, 1]$, $F=[0, 0.1]$, pick $x=0.85$).}

RTL: By Fact \ref{clutter}, we know that $\pcp$ obeys Id and Tol for any $\para$.
If $\pcp$ is also closed under $\neg$L, it follows that $Pa, a \sim_P b, \neg Pb \pcp$, so providing a counterexample to this would show that $\pcp$ isn't closed under $\neg$L.
Similarly, if $\pcp$ is closed under $\neg$R, it follows that $a \sim_P b \pcp Pb, \neg Pa$, so providing a counterexample to this would show that $\pcp$ isn't closed under $\neg$R.

So suppose $\para = \tuple{V, T, F}$ is not symmetric.
Then either there is $x \in T$ but $1 - x \not\in F$, or there is $x \in F$ but $1 - x \not\in T$.
Take a model $M$ with its domain containing just the terms $a$ and $b$; let the terms $a$ and $b$ denote themselves, and let all other constant terms denote $a$.

If there is $x \in T$ but $1 - x \not\in F$, then let $M(P)(a) = 1$, $M(P)(b) = 1 - x$, and $M(\sim_P)(a, b) = M(\sim_P)(b, a) = 1$.
For all other predicates, let them take everything to 1.
This clearly gives a $\para$ counterexample to $\sqq{Pa, a \sim_P b, \neg Pb\ }$; it is $\para$ tolerant as well, since the only real risk to tolerance comes from $a \sim_P b$, and while $M(Pa) \in T$, we do not have $M(Pb) \in F$.
So in this case $\pcp$ is not closed under $\neg$L.

Or if there is $y \in F$ but $1 - y \not \in T$, then let $M(P)(a) = y$, $M(P)(b) = 0$, and $M(\sim_P)(a, b) = M(\sim_P)(b, a) = 1$.
For all other predicates, let them take everything to 1.
This clearly gives a $\para$ counterexample to $\sqq{a \sim_Q b}{Pb, \neg Pa}$; it is $\para$ tolerant as well, since the only real risk to tolerance comes from $a \sim_Q b$, and while $M(Pb) \in F$, we do not have $M(Pa) \in T$. 
So in this case $\pcp$ is not closed under $\neg$R.

Similar arguments will show that in the first case $\pcp$ is not closed under $\to$L and in the second not under $\to$R.
(We have $M(\neg(a \sim_P a)) = 0$, so $A \to \neg(a \sim_P a)$ is equivalent to $\neg A$ on any model.)

\end{proof}

The negative metainferences addressed, only the eigenvariable metainferences remain.

\begin{defn} \label{opendefn}
A parameter $\para = \tuple{V, T, F}$ is {\em open} iff for all $X \subseteq V$, if $\glb{X} \in F$ then $F \cap X \neq \emptyset$, and if $\lub{X} \in T$ then $T \cap X \neq \emptyset$.\footnote{This condition on $T$ is that it be {\em Scott open} (see e.g.\ \cite[p.\ 95]{vickers:tvl}), and this condition on $F$ is the order-dual.}  The consequence relation $\pscp$ is {\em open} iff $\para$ is.
\end{defn}

Note that every parameter of the form $\tuple{[0, 1], (x, 1], [0, y)}$ is open, as is every parameter $\tuple{V, T, F}$ with finite $V$. But not all parameters are open; take $\tuple{[0,1], [.6, 1], [0, .4]}$. Let $X = \{x \in [0, 1] : x < .6\}$. Then $\lub X = .6 \in T$, but $T \cap X = \emptyset$.

\begin{fact} \label{eigenmi}
%If 
$\para$ is open {iff} $\pscp$ is closed under $\forall$R and $\exists$L.
\end{fact}

\begin{proof}
{LTR:} Let $\para = \tuple{V, T, F}$.
 Suppose $\G \not\pscp \forall x A(x), \Del$. Then there is some model $M$ with: $M(\gamma) \in T$ for all $\gamma \in \G$; $M(\delta) \in F$ for all $\delta \in \Del$; and $M(\forall x A(x)) \in F$. Since $M(\forall x A(x)) = \glb(\{M'(A(x)) : M' \text{ is an $x$-variant of } M\})$ and $\para$ is open, there must be some $x$-variant $M'$ of $M$ such that $M'(A(x)) \in F$. Consider now the $a$-variant $M''$ of $M$ such that $M''(a) = M'(x)$. Since $a$ is an eigenvariable, $M''$ matches $M$ on everything in $\G$ and $\Del$. So $\G \not\pscp A(x), \Del$, since $M''$ is a $\para$ tolerant countermodel.

$\exists$L is similar.\medskip

%\addPE{(RTL) Suppose $\para$ is not open. Either there is $X\subseteq V$ such that $\lub(X)\in T$ but $X\cap T=\emptyset$. Let $M$ be a model such that $\{M(P)(d);d\in D\}=X$. Clearly, for every $a$, $P(a)\pscp \bot$. However, $\exists x Px \not \pscp \bot$. For $M(\exists x Px)=\lub(\{M'(P(x); M'\text{ is an $x$-variant of } M\})\in T$. Hence $\exists L$ fails. Or there is $X\subseteq V$ such that $\glb(X)\in F$ but $X\cap F=\emptyset$. Let $M$ be a model such that $\{M(P)(d);d\in D\}=X$. Then $\top \pscp P(a)$ for any $a$. However, $\top \not \pscp \forall x Px$, hence $\forall R$ fails.}

RTL: By Fact \ref{clutter}, we know that $\pcp$ obeys Id and Tol for any $\para$.
If $\pcp$ is also closed under $\exists$L, it follows that $\exists y P y, \forall x (x \sim_P a) \pcp Pa$, so providing a counterexample to this would show that $\pcp$ isn't closed under $\exists$L.
Similarly, if $\pcp$ is closed under $\forall$R, it follows that $Pa, \forall x (a \sim_P x) \pcp \forall y Py$, so providing a counterexample to this would show that $\pcp$ isn't closed under $\forall$R.

So suppose $\para = \tuple{V, T, F}$ is not open.
Then there is $X \subseteq V$ with either: $\lub{X} \in T$ but $X \cap T = \emptyset$, or $\glb{X} \in F$ but $X \cap F = \emptyset$.
Take a model $M$ with domain $X \cup \{a\}$, for some $a \not\in X$; let every constant term of the language denote $a$.
For every $x \in X$, let $M(P)(x) = x$, let $M(\sim_P)$ take every pair of objects to 1, and let all other predicates take every object to 1.

If $\lub{X} \in T$ but $X \cap T = \emptyset$, then let $M(P)(a) = 0$.
In this case, $M$ is a $\para$ counterexample to $\sqq{\exists y P y, \forall x (x \sim_P a)}{Pa}$; it is $\para$ tolerant since there is no $z$ in the domain with $M(P)(z) \in T$.

If $\glb{X} \in F$ but $X \cap F = \emptyset$, then let $M(P)(a) = 1$.
In this case, $M$ is a $\para$ counterexample to $\sqq{Pa, \forall x (a \sim_P x)}{\forall y P y}$; it is $\para$ tolerant since there is no $z$ in the domain with $M(P)(z) \in F$.

\end{proof}

We now have a lot of pieces scattered around. Putting them together:

\begin{thm} \label{soparast}
For any proper symmetric open parameter $\para$, $\pscp \;=\; \STsc$.
\end{thm}

\begin{proof}
Fact \ref{parainst} gives us that $\G \pscp \Del$ implies $\G \STsc \Del$ for proper $\para$; it remains only to show the converse. Facts \ref{paratol}, \ref{clutter}, \ref{parami}, \ref{negmi}, and \ref{eigenmi} together establish that the sequent calculus in Figure \ref{stsimcalc} is sound for symmetric open $\pscp$. But since this calculus is complete for $\STsc$ (Fact \ref{stsc}), we're done.
\end{proof}

That is, there is only one proper symmetric open $\para$ tolerant consequence relation, and it is exactly the consequence relation of \STsim! The consequence relation of \STsim\ thus occupies a natural place among parameterised consequence relations. The steps needed to dodge sorites trouble (properness) and ensure operational classicality (symmetry and openness) also ensure that the consequence relation of \STsim\ is the only choice. As these are natural desiderata, \STsim\ looms large. 

Earlier, we identified Smith's approach as the parameterised consequence relation with the parameter S = $\tuple{[0, 1], (.5, 1], [0, .5)}$. Turning to its tolerant extension, we can see that this parameter is proper, symmetric, and open, so $\psc{\text{S}}$, like all such consequence relations, is the same as $\STsc$. We thus reckon that anyone interested in taking a Smith-style approach to consequence in a language including $\sim$ relations should arrive at \STsim\ as their desired consequence relation.

\section{How many truth values?} \label{discussion}

We have, in effect, presented many different model theories for the same consequence relation. 
Recall our earlier $n$-valued model theory for classical logic, based on the parameter $\tuple{V_n, \{1\}, \{0\}}$ with $V_n = \{0, 1/(n - 1), \ldots, (n - 2)/(n - 1), 1\}$.
This parameter is proper when $n \geq 3$, symmetric, and since $V_n$ is finite, it is also open. So we also have $n$-valued presentations of \STsim\ for all $n \geq 3$. Of course, we also have continuum-valued presentations, provided by parameters of the form $\tuple{[0, 1], (x, 1], [0, 1-x)}$, for $x \in [.5, 1)$; these are all proper, symmetric, and open, and Smith-consequence is among them, with $x = .5$. We can have countably-valued presentations as well. For example, let $V = \{x \in [0, 1] : \exists n \in \mathbb{N}$ s.t.$\ x = \frac{1}{2^n}$ or $x=\frac{2^n - 1}{2^n} \}$. That is, $V = \{0, \ldots, \frac{1}{16}, \frac{1}{8}, \frac{1}{4}, \frac{1}{2}, \frac{3}{4}, \frac{7}{8}, \frac{15}{16}, \ldots, 1\}$. $V$ is closed and countable. Now consider the parameter $\para = \tuple{V, [\frac{3}{4}, 1], [0, \frac{1}{4}]}$. This parameter is proper, symmetric, and open, and so subject to Theorem \ref{soparast}.\footnote{We can also see the original four-valued presentation of the strict-tolerant approach in \cite{cervr:tcs} through this lens; the appropriate parameter is $\tuple{\{0, x, 1 - x, 1\}, \{1\}, \{0\}}$, for any $0 < x < \half$.
These parameters too are proper, symmetric, and open, and so subject to Theorem \ref{soparast}.}

What to make of this situation? We argued in \S\ref{jolt} that three-valued approaches, especially those that validate tolerance, need not face the jolt problem, if appropriately interpreted.
We showed in \S\ref{paracons} that \STsim\ is a three-valued approach that validates tolerance, and that it also captures exactly the arguments valid in Smith's continuum-valued framework, when it is extended with similarity predicates. 
There are two purposes, then, for which Smith's continuum-many values are otiose: they are not needed to avoid the jolt problem, nor are they needed to do any logical work.
For each of these purposes, three values suffice. 

Earlier, we distinguished between the referential and the inferential roles that truth values play in a theory of meaning. Cast in terms of that distinction, we see that Smith's theory uses infinitely many values as references of sentences, but our argument shows that only three values are needed at the inferential level to capture his consequence relation. By contrast, the strict-tolerant theory basically uses three values at both levels, referential and inferential: on our approach, the third value is assigned to a vague predicate to capture borderline status, and the same third value is used to validate the tolerance principle. Smith argues that any finite assignment of values to vague sentences will involve `jolts' between distinct semantic values, but our contention is that a jolt-free approach is achieved first and foremost at the inferential level in the form of tolerance principles. In brief, the strict-tolerant account trims down Smith's referential apparatus to the minimum number of truth values needed at the inferential level.

%There is a clear and precise sense, therefore, in which the strict-tolerant approach is canonical for the kind of treatment of vagueness advocated by Smith.

This is not yet to say that three values will always be enough to satisfy any further desiderata you might consider for a theory of vagueness. Nor can we claim that an adequate theory of meaning is one that would necessarily use the same number of truth values at the referential and at the inferential level.
We close, then, by pausing to forestall some objections. 
First of all, you might note that the notion of validity we have taken from \cite{smith:vdt} is not a usual one for fuzzy treatments of vagueness.
This is certainly true; for other ways to proceed, see eg \cite{machina:vp, paoli:rfa,smith:cdss,paoli2019degrees}.
But our goal here has not been to provide an overview of fuzzy theories of vagueness; rather, it has been to explore the relations between a particular well-worked-out view---the one of \cite{smith:vdt}---and the nontransitive project advocated in our previous work.

% by Cobreros et al.

Touching on larger issues, you might worry that more than three values may be needed to account for the semantics of comparatives (taking ``$a$ is taller than $b$'' to imply that the degree of truth of ``$a$ is tall'' is greater than that of ``$b$ is tall'' \cite{paoli:rfa}), for the semantics of modifiers such as ``very'' or ``determinately'' \cite{lakoff:hedges}, or to model degrees of closeness to clear cases \cite{edgington:vbd, dd:wgm}. 
We haven't shown that three values are sufficient for such purposes, but that was not the goal; what counts as a minimal number of truth values for a fully adequate total semantic theory remains a broader issue.\footnote{Of these, the worry we have encountered most frequently is the one to do with comparatives. 
The topic is too broad for us to tackle in any detail here, but we will make two quick remarks.

First, we take it that comparative sentences involving gradable adjectives, such as ``A is younger than B'', can be treated without the need for additional truth values. 
Degrees of youth don't have to be degrees of truth. See in particular \cite{klein:spca} for an influential (three-valued) approach based on delineations -- with \cite{burnett:pccc} for a recent development--, and \cite{kennedy2007vagueness} for an influential (two-valued) approach based on scalar degrees (distinct from degrees of truth).

Second, even if ``true'' is \emph{itself} a gradable adjective (see \cite{henderson2021truth, egre2021half}), and so admits of degrees (which is surely contentious), it would not follow that there is any need to wire such degrees into the semantics of atomic predications in full generality.
}

For now, we have at least shown that some of the fuzziness in Smith's original approach can be shaved with Occam's razor.%

%\newpage

\begin{acknowledgement}
%We thank several audiences, in particular in Cerisy-la-Salle where the first sketch of this paper was given in 2011, in Opole (ESSLLI 2012), NYU (2013), as well as in Vienna (2014), and in Torino (2017). We wish to thank Christian Ferm\"uller, Kathrin Gl\"uer-Pagin, Peter Pagin, Jeff Paris, and Graham Priest for their encouraging feedback. 
This paper has gone through scattered periods of activity followed by long phases of sleep since its inception, and we are grateful to the editors for this opportunity to eventually publish it. Special thanks to audiences in Cerisy-la-Salle, where the first sketch of this paper was presented (Cerisy Context Conference, 2011), in Opole (ESSLLI 2012 workshop on Trivalent Logics and their Applications), NYU (Vagueness seminar held in 2013), Vienna (Theory and Logic Group's Research Seminar, 2014), Torino (Reasoning Club Conference, 2017), and Cagliari (Trends in Logic XXII, 2022). We thank four anonymous referees for comments on different occasions, including one referee for this volume. We also thank Christian Ferm\"uller, Kathrin Gl\"uer-Pagin, Francesco Paoli, Peter Pagin, Jeff Paris, and Graham Priest, for valuable exchanges at different stages of this work. The research leading to these results has received support from the Agence Nationale de la Recherche (grants TRILOGMEAN ANR-14-CE30-0010-01 and AMBISENSE ANR-19-CE28-0019-01), from the Department of Cognitive Studies of Ecole normale sup\'erieure (grant FRONTCOG ANR-17-EURE-0017), from the NWO funded `Language in Interaction' project, and from the ESSENCE-Marie Curie Initial Training network FP7-PEOPLE-2013-ITN), funded by the European Commission (grant agreement no. 607062). % Finally, we thank the editors for their invitation to contribute this paper, which alternated between long periods in a drawer and periods of revision.
\end{acknowledgement}

\bibliographystyle{apalike} % amsplain
\bibliography{tdt}

\begin{thebibliography}{}

\bibitem[Barrio et~al., 2015]{brt:lstl}
Barrio, E., Rosenblatt, L., and Tajer, D. (2015).
\newblock The logics of strict-tolerant logic.
\newblock {\em Journal of Philosophical Logic}, 44(5):551--571.

\bibitem[Beall and van Fraassen, 2003]{bf:pp}
Beall, J. and van Fraassen, B.~C. (2003).
\newblock {\em Possibilities and Paradox: An Introduction to Modal and
  Many-valued Logic}.
\newblock Oxford University Press, Oxford.

\bibitem[Blasio et~al., 2017]{blasio2017inferentially}
Blasio, C., Marcos, J., and Wansing, H. (2017).
\newblock An inferentially many-valued two-dimensional notion of entailment.
\newblock {\em Bulletin of the Section of Logic}, 46(3/4):233--262.

\bibitem[Borel, 1907]{borel2014economic}
Borel, {\'E}. (2014 (1907)).
\newblock An economic paradox: The sophism of the heap of wheat and statistical
  truths.
\newblock {\em Erkenntnis}, 79(5):1081--1088.
\newblock Originally published in French in \textit{La Revue du Mois}.

\bibitem[Burnett, 2014]{burnett:pccc}
Burnett, H. (2014).
\newblock Penumbral connections in comparative constructions.
\newblock {\em Journal of Applied Non-Classical Logics}, 24(1--22):35--60.

\bibitem[Chemla and Egr{\'e}, 2019]{chemla2019suszko}
Chemla, E. and Egr{\'e}, P. (2019).
\newblock Suszko's problem: Mixed consequence and compositionality.
\newblock {\em The Review of Symbolic Logic}, 12(4):736--767.

\bibitem[Cobreros, 2013]{cobreros:vs}
Cobreros, P. (2013).
\newblock Vagueness: Subvaluationism.
\newblock {\em Philosophy Compass}, 8(5):472--485.

\bibitem[Cobreros et~al., 2012]{cervr:tcs}
Cobreros, P., {\'{E}}gr{\'{e}}, P., Ripley, D., and van Rooij, R. (2012).
\newblock Tolerant, classical, strict.
\newblock {\em Journal of Philosophical Logic}, 41(2):347--385.

\bibitem[Cobreros et~al., 2013]{cervr:rtt}
Cobreros, P., {\'{E}}gr{\'{e}}, P., Ripley, D., and van Rooij, R. (2013).
\newblock Reaching transparent truth.
\newblock {\em Mind}, 122(488):841--866.

\bibitem[Cobreros et~al., 2015a]{cervr:pive}
Cobreros, P., {\'{E}}gr{\'{e}}, P., Ripley, D., and van Rooij, R. (2015a).
\newblock Pragmatic interpretations of vague expressions: Strongest meaning and
  nonmonotonic consequence.
\newblock {\em Journal of Philosophical Logic}, 44(4):375--393.

\bibitem[Cobreros et~al., 2015b]{cervr:vtpc}
Cobreros, P., {\'{E}}gr{\'{e}}, P., Ripley, D., and van Rooij, R. (2015b).
\newblock Vagueness, truth, and permissive consequence.
\newblock In Achourioti, T., Galinon, H., Fujimoto, K., and Fern{\'{a}}ndez,
  J.~M., editors, {\em Unifying the Philosophy of Truth}, pages 409--430.
  Springer.

\bibitem[Decock and Douven, 2014]{dd:wgm}
Decock, L. and Douven, I. (2014).
\newblock What is graded membership?
\newblock {\em No{\^u}s}, 48(4):653--682.

\bibitem[Edgington, 1997]{edgington:vbd}
Edgington, D. (1997).
\newblock Vagueness by degrees.
\newblock In Keefe, R. and Smith, P., editors, {\em {V}agueness: {A} {R}eader},
  pages 294--316. {MIT} {P}ress, Cambridge, Massachusetts.

\bibitem[{\'{E}}gr{\'{e}}, 2011a]{egre:pas}
{\'{E}}gr{\'{e}}, P. (2011a).
\newblock Perceptual ambiguity and the sorites.
\newblock In Nouwen, R., van Rooij, R., Schmitz, H.-C., and Sauerland, U.,
  editors, {\em Vagueness and Communication}, pages 64--90. Springer.

\bibitem[{\'{E}}gr{\'{e}}, 2011b]{egre:vdtreview}
{\'{E}}gr{\'{e}}, P. (2011b).
\newblock Review of {N}icholas {S}mith, {{\em {V}agueness and {D}egrees of
  {T}ruth}}.
\newblock {\em Australasian Journal of Philosophy}, 89(1):177--180.

\bibitem[{\'E}gr{\'e}, 2017]{egre2017vague}
{\'E}gr{\'e}, P. (2017).
\newblock Vague judgment: a probabilistic account.
\newblock {\em Synthese}, 194(10):3837--3865.

\bibitem[\'Egr\'e, 2021]{egre2021half}
\'Egr\'e, P. (2021).
\newblock Half-truths and the liar.
\newblock In Nicolai, C. and Stern, J., editors, {\em Modes of truth}, pages
  18--40. Routledge.

\bibitem[{\'{E}}gr{\'{e}} and Barberousse, 2014]{eb:borel}
{\'{E}}gr{\'{e}}, P. and Barberousse, A. (2014).
\newblock Borel on the heap.
\newblock {\em Erkenntnis}, 79:1043--1079.

\bibitem[Fara, 2000]{fara:ss}
Fara, D.~G. (2000).
\newblock Shifting sands: An interest-relative theory of vagueness.
\newblock {\em Philosophical Topics}, 28(1):45--81.
\newblock (Originally published under the name ``Delia Graff'').

\bibitem[Fara, 2001]{fara:pcas}
Fara, D.~G. (2001).
\newblock Phenomenal continua and the sorites.
\newblock {\em Mind}, 110(440):905--935.
\newblock (Originally published under the name ``Delia Graff'').

\bibitem[Field, 2008]{field:stp}
Field, H. (2008).
\newblock {\em Saving Truth from Paradox}.
\newblock Oxford University Press, Oxford.

\bibitem[Frankowski, 2004]{frankowski:fpi}
Frankowski, S. (2004).
\newblock Formalization of a plausible inference.
\newblock {\em Bulletin of the Section of Logic}, 33:41--52.

\bibitem[French and Ripley, 2019]{french2019valuations}
French, R. and Ripley, D. (2019).
\newblock Valuations: bi, tri, and tetra.
\newblock {\em Studia Logica}, 107(6):1313--1346.

\bibitem[Gaifman, 2010]{gaifman:vtcl}
Gaifman, H. (2010).
\newblock Vagueness, tolerance and contextual logic.
\newblock {\em Synthese}, 174(1):5--46.

\bibitem[G\'omez-Torrente, 2011]{gt:slpdpv}
G\'omez-Torrente, M. (2011).
\newblock The sorites, linguistic preconceptions, and the dual picture of
  vagueness.
\newblock In Dietz, R. and Moruzzi, S., editors, {\em Cuts and Clouds}, pages
  228--253. Oxford University Press.

\bibitem[Haack, 1996]{haack1996deviant}
Haack, S. (1996).
\newblock {\em Deviant Logic, Fuzzy Logic}.
\newblock The University of Chicago Press: Chicago.

\bibitem[Henderson, 2021]{henderson2021truth}
Henderson, J. (2021).
\newblock Truth and gradability.
\newblock {\em Journal of Philosophical Logic}, 50(4):755--779.

\bibitem[Humberstone, 1988]{humberstone:hl}
Humberstone, L. (1988).
\newblock Heterogeneous logic.
\newblock {\em Erkenntnis}, 29(3):395--435.

\bibitem[Kamp, 1981]{kamp:poh}
Kamp, H. (1981).
\newblock The paradox of the heap.
\newblock In M{\"{o}}nnich, U., editor, {\em Aspects of Philosophical Logic},
  pages 225--277. D.\ Reidel.

\bibitem[Kennedy, 2007]{kennedy2007vagueness}
Kennedy, C. (2007).
\newblock Vagueness and grammar: The semantics of relative and absolute
  gradable adjectives.
\newblock {\em Linguistics and philosophy}, 30(1):1--45.

\bibitem[Klein, 1980]{klein:spca}
Klein, E. (1980).
\newblock A semantics for positive and comparative adjectives.
\newblock {\em Linguistics and philosophy}, 4(1):1--45.

\bibitem[Lakoff, 1973]{lakoff:hedges}
Lakoff, G. (1973).
\newblock Hedges: A study in meaning criteria and the logic of fuzzy concepts.
\newblock {\em Journal of Philosophical Logic}, 2(4):458--508.

\bibitem[Lassiter, 2011]{lassiter2011vagueness}
Lassiter, D. (2011).
\newblock Vagueness as probabilistic linguistic knowledge.
\newblock In {\em International workshop on vagueness in communication}, pages
  127--150. Springer.

\bibitem[Lassiter and Goodman, 2017]{lassiter2017adjectival}
Lassiter, D. and Goodman, N.~D. (2017).
\newblock Adjectival vagueness in a bayesian model of interpretation.
\newblock {\em Synthese}, 194(10):3801--3836.

\bibitem[{\L}ukasiewicz, 1920]{luk1920}
{\L}ukasiewicz, J. (1920).
\newblock [{O}n three-valued logic].
\newblock {\em Ruch Filosoficzny}, 5:70--71.
\newblock English translation in J. \L ukasiewicz, Selected Works, L. Borkowski
  ed., 1970.

\bibitem[Machina, 1972]{machina:vp}
Machina, K.~F. (1972).
\newblock Vague predicates.
\newblock {\em American Philosophical Quarterly}, 9(3):225--233.

\bibitem[Malinowski, 1990]{malinowski1990q}
Malinowski, G. (1990).
\newblock Q-consequence operation.
\newblock {\em Reports on mathematical logic}, 24:49--54.

\bibitem[Malinowski, 2009a]{malinowski2009beyond}
Malinowski, G. (2009a).
\newblock Beyond three inferential values.
\newblock {\em Studia Logica}, 92(2):203--213.

\bibitem[Malinowski, 2009b]{malinowski:cilmv}
Malinowski, G. (2009b).
\newblock Concerning intuitions on logical many-valuedness.
\newblock {\em Bulletin of the Section of Logic}, 38(3/4):111--121.

\bibitem[Malinowski, 2004]{malinowski2004strawsonian}
Malinowski, J. (2004).
\newblock Strawsonian presuppositions and logical entailment.
\newblock {\em Logique et Analyse}, pages 123--138.

\bibitem[Manor, 2006]{manor:solving}
Manor, R. (2006).
\newblock Solving the heap.
\newblock {\em Synthese}, 153(2):171--186.

\bibitem[Pagin, 2011]{pagin:vdr}
Pagin, P. (2011).
\newblock Vagueness and domain restriction.
\newblock In {\'{E}}gr{\'{e}}, P. and Klinedinst, N., editors, {\em Vagueness
  and Language Use}, pages 283--307. Palgrave.

\bibitem[Paoli, 2003]{paoli:rfa}
Paoli, F. (2003).
\newblock A really fuzzy approach to the sorites paradox.
\newblock {\em Synthese}, 134(3):363--387.

\bibitem[Paoli, 2019]{paoli2019degrees}
Paoli, F. (2019).
\newblock Degree theories and the sorites paradox.
\newblock In Oms, S. and Zardini, E., editors, {\em The Sorites Paradox}, pages
  151--167. Cambridge University Press.

\bibitem[Paris et~al., 2009]{paris2009inconsistency}
Paris, J.~B., Muino, D.~P., and Rosefield, M. (2009).
\newblock Inconsistency as qualified truth: A probability logic approach.
\newblock {\em International Journal of Approximate Reasoning},
  50(8):1151--1163.

\bibitem[Priest, 2008]{priest:intro}
Priest, G. (2008).
\newblock {\em An Introduction to Non-Classical Logic: From If to Is}.
\newblock Cambridge University Press, Cambridge, 2nd edition.

\bibitem[Priest, 2013]{priest:vi}
Priest, G. (2013).
\newblock Vague inclosures.
\newblock In \cite{tbmp:pla}, pages 367--378.

\bibitem[Raffman, 1996]{raffman:vcr}
Raffman, D. (1996).
\newblock Vagueness and context-relativity.
\newblock {\em Philosophical Studies}, 81:175--192.

\bibitem[Read and Wright, 1985]{rw:hpt}
Read, S. and Wright, C. (1985).
\newblock Hairier than {P}utnam thought.
\newblock {\em Analysis}, 45(1):56--58.

\bibitem[Ripley, 2013a]{ripley:pafc}
Ripley, D. (2013a).
\newblock Paradoxes and failures of cut.
\newblock {\em Australasian Journal of Philosophy}, 91(1):139--164.

\bibitem[Ripley, 2013b]{ripley:sos}
Ripley, D. (2013b).
\newblock Sorting out the sorites.
\newblock In \cite{tbmp:pla}, pages 329--348.

\bibitem[Rotstein and Winter, 2004]{rotstein2004total}
Rotstein, C. and Winter, Y. (2004).
\newblock Total adjectives vs. partial adjectives: Scale structure and
  higher-order modifiers.
\newblock {\em Natural language semantics}, 12(3):259--288.

\bibitem[Rumfitt, 2015]{rumfitt:bst}
Rumfitt, I. (2015).
\newblock {\em The Boundary Stones of Thought}.
\newblock Oxford University Press, Oxford.

\bibitem[Scharp, 2013]{scharp:rt}
Scharp, K. (2013).
\newblock {\em Replacing Truth}.
\newblock Oxford University Press, Oxford.

\bibitem[Shramko and Wansing, 2011]{shramko2011truth}
Shramko, Y. and Wansing, H. (2011).
\newblock {\em Truth and falsehood: An inquiry into generalized logical
  values}, volume~36.
\newblock Springer, Trends in Logic.

\bibitem[Smith, 2005]{smith:vac}
Smith, N. J.~J. (2005).
\newblock Vagueness as closeness.
\newblock {\em Australasian Journal of Philosophy}, 83(2):157--183.

\bibitem[Smith, 2008]{smith:vdt}
Smith, N. J.~J. (2008).
\newblock {\em Vagueness and Degrees of Truth}.
\newblock Oxford University Press, Oxford.

\bibitem[Smith, 2016]{smith:cdss}
Smith, N. J.~J. (2016).
\newblock Consonance and dissonance in solutions to the sorites.
\newblock In Bueno, O. and Abasnezhad, A., editors, {\em On the Sorites
  Paradox}. Springer.
\newblock Forthcoming.

\bibitem[Sorensen, 2001]{sorensen:vac}
Sorensen, R. (2001).
\newblock {\em Vagueness and Contradiction}.
\newblock Oxford University Press, Oxford.

\bibitem[Suszko, 1977]{suszko1977fregean}
Suszko, R. (1977).
\newblock The {F}regean axiom and polish mathematical logic in the 1920 s.
\newblock {\em Studia Logica}, 36(4):377--380.

\bibitem[Takeuti, 1987]{takeuti:pt}
Takeuti, G. (1987).
\newblock {\em Proof Theory}.
\newblock Elsevier Science, Amsterdam, 2nd edition.
\newblock Studies in Logic 81.

\bibitem[Tanaka et~al., 2013]{tbmp:pla}
Tanaka, K., Berto, F., Mares, E., and Paoli, F., editors (2013).
\newblock {\em Paraconsistency: Logic and Applications}.
\newblock Springer, Dordrecht.

\bibitem[Tsuji, 1998]{tsuji:suszko}
Tsuji, M. (1998).
\newblock Many-valued logics and {S}uszko's {T}hesis revisited.
\newblock {\em Studia Logica}, 60(2):299--309.

\bibitem[van Rooij, 2011]{rooij2011implicit}
van Rooij, R. (2011).
\newblock Implicit versus explicit comparatives.
\newblock In \'Egr\'e, P. and Klinedinst, N., editors, {\em Vagueness and
  Language Use}, pages 51--72. Palgrave Macmillan.

\bibitem[Vickers, 1989]{vickers:tvl}
Vickers, S. (1989).
\newblock {\em Topology via Logic}.
\newblock Cambridge University Press, Cambridge.

\bibitem[Weber et~al., 2014]{wrphc:tg}
Weber, Z., Ripley, D., Priest, G., Hyde, D., and Colyvan, M. (2014).
\newblock Tolerating gluts: a reply to {Beall}.
\newblock {\em Mind}, 123(491):813--828.

\bibitem[Williamson, 1994]{williamson:v}
Williamson, T. (1994).
\newblock {\em Vagueness}.
\newblock Routledge, London.

\bibitem[Wright, 1975]{wright:ocvp}
Wright, C. (1975).
\newblock On the coherence of vague predicates.
\newblock {\em Synthese}, 30(3):325--365.

\bibitem[Wright, 2001]{wright:quandary}
Wright, C. (2001).
\newblock On being in a quandary: Relativism, vagueness, logical revisionism.
\newblock {\em Mind}, 110(437):45--98.

\bibitem[Wright, 2003]{wright:vfca}
Wright, C. (2003).
\newblock Vagueness: a fifth column approach.
\newblock In Beall, J., editor, {\em Liars and Heaps: New Essays on Paradox},
  pages 84--105. Oxford University Press, Oxford.

\bibitem[Zardini, 2008]{zardini:mot}
Zardini, E. (2008).
\newblock A model of tolerance.
\newblock {\em Studia Logica}, 90(3):337--368.

\end{thebibliography}

\end{document}